\documentclass[11pt,pdftex]{amsart}
\usepackage{amsmath,amssymb,amsthm,amscd}
\usepackage{txfonts,mathptmx}
\usepackage{mathrsfs}
\usepackage[colorlinks=true]{hyperref}
\usepackage[msc-links,backrefs]{amsrefs}
\usepackage{tikz-cd}

\usepackage{comment}

\numberwithin{equation}{section}

\newtheorem{thm}{Theorem}[section]
\newtheorem*{thm-nn}{Theorem}

\newtheorem{lem}[thm]{Lemma}
\newtheorem{pro}[thm]{Proposition}
\newtheorem*{pro-nn}{Proposition}
\newtheorem{cor}[thm]{Corollary}
\newtheorem*{cor-nn}{Corollary}
\newtheorem*{conj-nn}{Conjecture}

\theoremstyle{definition}
\newtheorem{defn}[thm]{Definition}
\newtheorem{eg}[thm]{Example}

\theoremstyle{remark}

\newtheorem{rem}[thm]{Remark}
\newtheorem*{ack}{Acknowledgement}

\newcommand{\As}{{\operatorname{As}}}
\newcommand{\Conj}{{\operatorname{Conj}}}

\newcommand{\bare}{s}

\newcommand{\R}{\mathbb{R}}
\newcommand{\Z}{\mathbb{Z}}

\newcommand{\Image}{{\operatorname{Im}}}

\newcommand{\Hom}{\operatorname{Hom}}

\newcommand{\Inv}{\operatorname{Inv}}

\newcommand{\Ab}{\mathrm{Ab}}
\newcommand{\Gp}{\mathbf{Grp}}
\newcommand{\Qdl}{\mathbf{Qdl}}
\newcommand{\SQ}{\mathbf{SQ}}

\begin{document}
\title{Associated groups of symmetric quandles}
\author{Toshiyuki Akita}
\address{Department of Mathematics, Faculty of Science, Hokkaido University,
Sapporo, 060-0810 Japan}
\email{akita@math.sci.hokudai.ac.jp}
\author{Kakeru Shikata}
\address{Department of Mathematics, Graduate School of Science, Hokkaido University,
Sapporo, 060-0810 Japan}
\email{shikata.kakeru.s5@elms.hokudai.ac.jp}


\begin{abstract}
A symmetric quandle is a quandle equipped with a good involution.
In this paper, we investigate the structure of the associated groups of symmetric quandles.
Among other results, we examine the relationship between the associated group of a symmetric quandle and that of its underlying quandle.
As a consequence, we obtain three distinct characterizations of the associated group of the underlying quandle in terms of that of the symmetric quandle.
We also provide a group-theoretic characterization of the associated group of a symmetric quandle and compute its abelianization, showing that it is isomorphic to the first symmetric quandle homology.
Furthermore, we express the second quandle homology of a quandle in terms of the associated group of the corresponding symmetric quandle.
Finally, we show that a symmetric quandle is embeddable in its associated group if and only if its underlying quandle is.
\end{abstract}

\maketitle


\section{Introduction}\label{sec:Intro}
\subsection*{Background}
A \emph{quandle} is a set $Q$ equipped with a binary operation $(x,y) \mapsto x \ast y$ 
satisfying the following axioms:
\begin{enumerate}
\item $x\ast x=x\quad (x\in Q)$,
\item For any $x\in Q$, the map $S_x\colon Q\to Q$ defined by $y\mapsto y\ast x$ is bijective,
\item $(x\ast y)\ast z=(x\ast z)\ast (y\ast z)\quad (x,y,z\in Q).$
\end{enumerate}
If $Q$ satisfies axioms (2) and (3),  but not necessarily (1), then $Q$ is called a \emph{rack}.
A fundamental example of a quandle is the \emph{conjugation quandle}
$\Conj(G)$ of a group $G$, whose underlying set is $G$,
and whose quandle operation is given by the conjugation $g\ast h\coloneqq h^{-1}gh$.

Quandles were independently defined 
by Joyce \cite{MR0638121} and Matveev \cite{MR0672410}, 
both motivated by knot theory. 
Among other results, they showed that the ``knot quandle'' serves as a complete invariant of knots. 
Since then, quandles have found numerous applications across various areas of mathematics, 
including low-dimensional topology \cites{MR3524589,MR3821082}, 
set-theoretic solutions to the Yang--Baxter equation \cites{MR3974961,arXiv:2506.23175}, 
Hopf algebras \cite{MR1994219}, 
symmetric spaces \cite{MR3127819}, 
and more.
In particular, coloring invariants and cocycle invariants, the latter of which was first introduced in \cite{MR1990571}, constructed via quandles have achieved remarkable success in the study of knots and surface-knots in low-dimensional topology.

The \emph{associated group} of a quandle $Q$  (also known as the \emph{adjoint group}, \emph{enveloping group}, 
or \emph{structure group} in the literature),
denoted by $\As(Q)$, is the group defined by the presentation
\[
\As(Q)\coloneqq\langle e_x\ (x\in Q)\mid e_y^{-1}e_x e_y=e_{x\ast y}
\ (x,y\in Q)\rangle.
\]
There are several reasons why the associated groups are relevant.
Among them, one is that the functor $Q\mapsto\As(Q)$, from the category of quandles to
the category of groups, is left adjoint to the functor $G\mapsto\Conj(G)$,
 which assigns to each group its associated conjugation quandle.
Second, although quandle homology groups, introduced in \cite{MR1990571}, are generally difficult to compute, Eisermann's formula \cite{MR3205568} provides a method for computing the second homology group when the associated group is sufficiently well understood (see \S\ref{sec:2nd-qdl-homology}).
Third, it was proved in \cite{MR4669143} that
a group $G$ is isomorphic to the associated group of some quandle if and only if
$G$ admits a presentation $\langle X\mid R\rangle$ in which every
defining relation is of the form
\begin{equation}\label{eq:Wirtinger-ours}
w^{-1}xw=y\quad (x,y\in X, w\in F(X)),
\end{equation}
where $F(X)$ denotes the free group on a set $X$.
We refer to such a presentation as a \emph{Wirtinger presentation} in this paper, 
as it generalizes the classical Wirtinger presentations of knot groups.
Groups admitting a Wirtinger presentation include free groups and free abelian groups, 
knot groups and link groups, high-dimensional knot and link groups,
braid groups and pure braid groups, Artin groups, 
Thompson's group $F$, free central extensions of arbitrary 
groups \cite{MR1356352}---including free nilpotent groups---among others.
All of them arise as associated groups of some quandles.

On the other hand, associated groups are often difficult to compute in practice, since they are defined via group presentations. The difficulty becomes more severe when the quandle is infinite, 
as the defining presentation then necessarily involves infinitely many generators and relations.
A successful approach to studying associated groups is to add power relations of the form 
$(e_x)^{n(x)}=1$, where 
$n(x)$ is a carefully chosen nonnegative integer for each $x\in Q$,
to their presentations, thereby constructing more manageable quotient groups \cites{MR4447657,hasegawa-thesis,MR3974961,MR3671570,MR2803792,MR3276225,arXiv:2407.02955}.
One of the main goals of this paper is to demonstrate that the associated group of a symmetric quandle itself serves as a manageable quotient group—comparable to those obtained by adding power relations.

A symmetric quandle is a quandle equipped with a good involution.
An involution $\rho\colon Q \to Q$ on a quandle $Q$ 
is said to be a \emph{good involution} if it satisfies:
\begin{enumerate}
\item $\rho(x\ast y)=\rho(x)\ast y\quad (x,y\in Q),$
\item $x\ast \rho(y)=S_y^{-1}(x)\quad  (x,y\in Q).$
\end{enumerate}
A \emph{symmetric quandle} $(Q,\rho)$ is a pair of a quandle $Q$
and a good involution $\rho\colon Q\to Q$.
A symmetric rack is defined similarly.
A fundamental example of a symmetric quandle is the
\emph{conjugation symmetric quandle} $(\Conj(G),\mathrm{Inv}(G))$
of a group $G$, where $\mathrm{Inv}(G)$ is the involution defined by $g\mapsto g^{-1}$.

Symmetric quandles, together with their associated groups and symmetric quandle (co)homology,  
were introduced by Kamada \cite{MR2371714} (see also Kamada--Oshiro \cite{MR2657689}) 
in order to define invariants of links and surface-links 
that are not necessarily oriented or orientable.
Since then, symmetric quandles have been extensively studied and applied in low-dimensional topology
\cites{MR2593699,MR3363811,MR2945646,MR2821435,MR2629767,
MR2890467,MR2796229,MR4312970,MR4766150,MR3402903,arXiv:2010.09983,MR4594919}, 
and have also been investigated from  non-topological perspective
 \cites{MR4881590,MR4441473,MR3065997,MR4688855,arXiv:2412.20081,arXiv:2505.08090,
 MR3079760,arXiv:2407.02971,MR4899477}.
 
 The main objectives of this paper are twofold:
 (1) We demonstrate that the associated groups $\As(Q,\rho)$ of symmetric quandles $(Q,\rho)$,
 which are defined by the group presentation
 \[
\As(Q,\rho)\coloneqq\langle \bare_x\ (x\in Q)\mid \bare_y^{-1}\bare_x \bare_y
=\bare_{x\ast y}, \ \bare_{\rho(x)}=\bare_x^{-1}\ (x,y\in Q)\rangle,
\]
 are as significant and mathematically rich as those of ordinary quandles.
 (2) Given a symmetric quandle $(Q,\rho)$, we show that the structure of the associated group $\As(Q)$ of the
 underlying quandle $Q$ can, in several ways, be characterized 
 via the associated group $\As(Q,\rho)$ of the  symmetric quandle.

\subsection*{The associated groups of underlying quandles}
Let us introduce our results concerning the relationship
between the associated groups of symmetric quandles
and those of the underlying quandles.
Let $(Q,\rho)$ be a symmetric quandle.
Let $\pi_Q\colon \As(Q)\to\As(Q,\rho)$ be the canonical surjection
defined by $e_x\mapsto s_x$ $(x\in Q)$.
Let $Z(Q,\rho)$ denote its kernel.
First, we prove that $Z(Q,\rho)$ lies in the center of $\As(Q)$ and is a free abelian group
 with an explicitly described basis (Theorem \ref{thm:main}).
Consequently, they fit into a central extension with a free abelian kernel:
\begin{equation}\label{eq:intro-central-ext}
1\to Z(Q,\rho)\to\As(Q)\xrightarrow{\pi_Q}\As(Q,\rho)\to 1.
\end{equation}
Thus, the study of the structure of $\As(Q)$ can, in principle, be reduced to the
study of $\As(Q,\rho)$ and a $2$-cocycle representing the class in $H^2(\As(Q,\rho);Z(Q,\rho))$
corresponding to the central extension \eqref{eq:intro-central-ext}.
Second, the canonical surjection $\pi_Q$ induces an isomorphism of commutator
subgroups (Corollary \ref{cor:commutator}):
\[
\pi_Q\colon [\As(Q),\As(Q)]\xrightarrow{\cong} [\As(Q,\rho),\As(Q,\rho)].
\]
Hence, there exists a group extension of the form
\begin{equation}\label{eq:intro-comm-seq}
1\to [\As(Q,\rho),\As(Q,\rho)]\to \As(Q)\to \As(Q)_\Ab\to 1.
\end{equation}
The abelianization of $\As(Q)$ is known to be free abelian (see Proposition \ref{pro:asq-abelianization}).
In particular, if $\As(Q)_\Ab\cong\Z$ then the extension \eqref{eq:intro-comm-seq} splits, and 
we obtain an isomorphism 
\[\As(Q)\cong [\As(Q,\rho),\As(Q,\rho)]\rtimes\Z.\]
Third, the associated groups of a symmetric quandle and that of the underlying quandle
form a pullback diagram. More precisely, the following commutative square
\begin{equation}\label{eq:intro-pullback}
\begin{tikzcd}
\As(Q)\arrow{r}{\Ab} \arrow[d,"\pi_Q"']
& \As(Q)_{\Ab}\arrow{d}{(\pi_Q)_\Ab}\\
\As(Q,\rho)\arrow{r}{\Ab}& \As(Q,\rho)_{\Ab}
\end{tikzcd}
\end{equation}
is a pullback diagram, where $\Ab$ denotes the abelianization map,
and the vertical arrows are the canonical surjection $\pi_Q$ and the induced map, respectively
(Theorem \ref{thm:pullback}).
Note that the right-hand side of the square is easy to understand since $\As(Q)_\Ab$ is free abelian and 
$\As(Q,\rho)_\Ab$ is a direct sum of copies of $\Z$ and $\Z/2\Z$ (Theorem \ref{thm:abelianization}).
In particular, the homomorphism 
\[
\As(Q)\to\As(Q,\rho)\times\As(Q)_\Ab,\quad
g\mapsto (\pi_Q(g),[g])
\]
is injective.
Moreover, the induced square of Eilenberg-MacLane spaces
\[\begin{tikzcd}
K(\As(Q),1)\arrow{r} \arrow[d]
& K(\As(Q)_{\Ab},1)\arrow{d}\\
K(\As(Q,\rho),1)\arrow{r}& K(\As(Q,\rho)_{\Ab},1)
\end{tikzcd}\]
is a homotopy pullback (Theorem \ref{thm:pullback}).
The right-hand side of the square is easy to understand, since
 $K(\Z,1)\simeq S^1$ and $K(\Z/2\Z,1)\simeq \R P^{\infty}$, as is well-known.

In conclusion, given a symmetric quandle,
we obtain three mutually different characterizations of the associated group
of the underlying quandle---namely, \eqref{eq:intro-central-ext},
\eqref{eq:intro-comm-seq}, and \eqref{eq:intro-pullback}---each 
of which fundamentally relies on the associated group of the symmetric quandle.
In addition, we express the second quandle homology $H_2(Q)$ of the underlying quandle $Q$
in terms of the associated group of the corresponding symmetric quandle $(Q,\rho)$
(Corollary \ref{cor:homology}).

\subsection*{The associated groups of symmetric quandles}
Now let us introduce our results concerning the structure of the associated groups of
symmetric quandles.
First of all, we provide a group-theoretic characterization of the associated groups of symmetric quandles.
Namely, a group $G$ is the associated group of some symmetric quandle if and only if
$G$ admits a group presentation $\langle X\mid R\rangle$ in which every
defining relation is of the form
\[
w^{-1}xw=y^\epsilon\quad (x,y\in X, w\in F(X),\epsilon\in\{\pm 1\}),
\]
where $F(X)$ is the free
group on $X$ (Corollary \ref{cor:char-As}).
We refer to such a presentation as a \emph{twisted Wirtinger presentation} in this paper, 
as it generalizes Wirtinger presentations \eqref{eq:Wirtinger-ours}.
In addition to groups admitting Wirtinger presentations mentioned above,
examples of groups admitting twisted Wirtinger presentations
include Coxeter groups,
 virtual braid groups,
loop braid groups (also known as welded braid groups),
 cactus groups,
and twisted Artin groups  \cite{MR2672155}.
Twisted Artin groups include, for example,
the fundamental group of the Klein bottle and
Baumslag--Solitar groups $B(k,k+1)$ for $k\geq 1$.

Further, it was shown in \cites{MR3363811,MR1841759} that a
group $G$ admits a \emph{finite} twisted Wirtinger presentation if and only if
there exists a closed $n$-dimensional manifold smoothly embedded in $\R^{n+2}$
whose complement has fundamental group isomorphic to $G$ (see Example \ref{eg:submfd}).
Moreover, in this situation, a group $G$ admits a Wirtinger presentation 
if and only if the manifold is orientable.
All such groups arise as associated groups of some symmetric quandles.
Conversely, we believe that groups admitting twisted Wirtinger presentations can be studied from the viewpoint of the associated groups of symmetric quandles.

Second, we compute the abelianization of the associated group of any symmetric quandle
 (Theorem~\ref{thm:abelianization}) and show that it is isomorphic to the first homology group of the symmetric quandle (Proposition~\ref{cor:1st-symm-qdl-homology}).
Third, from a categorical point of view, we show that the functor $(Q,\rho)\mapsto \As(Q,\rho)$,
from the category of symmetric quandles to the category of groups, is left adjoint to
the functor $G\mapsto (\Conj(G),\Inv(G))$, which assigns
to each group its associated conjugation symmetric quandle (Proposition \ref{pro:adjoint}).

Finally, we prove that a symmetric quandle $(Q,\rho)$ is embeddable in $\As(Q,\rho)$ if and only
if the underlying quandle $Q$ is embeddable in $\As(Q)$.
In other words,
the natural map $Q\to\As(Q,\rho)$ given by $x\mapsto s_x$ is injective 
if and only if the natural map $Q\to\As(Q)$ given by $x\mapsto e_x$
is injective  (Theorem \ref{thm:emeddable}).

Prior to this work, little was known about the structure of the associated groups of symmetric quandles or their relationship with the associated groups of the underlying quandles.
This paper presents the first comprehensive study that covers the associated groups of \emph{all} symmetric quandles.

The paper is organized as follows.
Let $\As(Q,\rho)$ denote the associated group of a symmetric quandle $(Q,\rho)$,
and let $\As(Q)$ denote the associated group of the underlying quandle $Q$.
Section \ref{sec:Prelim} provides preliminaries.
In Section \ref{sec:central-ext}, we characterize $\As(Q)$ as a central extension
of $\As(Q,\rho)$ with a free abelian kernel.
Section \ref{sec:pullback} shows that $\As(Q)$ and $\As(Q,\rho)$ 
fit into a pullback diagram.
In Section \ref{sec:group-theoretic-char}, we give a group-theoretic
characterization of $\As(Q,\rho)$.
In Section \ref{sec:abelianization}, we compute the abelianization of $\As(Q,\rho)$
and show that the first symmetric quandle homology of $(Q,\rho)$ is isomorphic to
$\As(Q,\rho)_\Ab$.
Section \ref{sec:adjoint} establishes that the functor $(Q,\rho)\mapsto\As(Q,\rho)$
from the category of symmetric quandles to the category of groups
is left adjoint to the functor $G\mapsto (\Conj(G),\Inv(G))$ from
the category of groups to the category of symmetric quandles.
In Section \ref{sec:2nd-qdl-homology}, we express the second quandle homology of 
a connected quandle $Q$ in terms of 
$\As(Q,\rho)$. 
Section \ref{sec:involutive} applies our results to the case of involutive quandles.
Finally, in Section \ref{sec:embedd}, we prove that a symmetric quandle $(Q,\rho)$
is embeddable in $\As(Q,\rho)$ if and only if the underlying quandle $Q$
is embeddable in $\As(Q)$.

\subsection*{Notation}
For a group $G$, let $G_{\mathrm{Ab}}\coloneqq G/[G,G]$ be the abelianization of $G$.
For an element of $g\in G$, we write $[g]\coloneqq g[G,G]\in G_{\mathrm{Ab}}$.
For a set $X$, the free group on $X$ is denoted by $F(X)$.

\section{Preliminaries}\label{sec:Prelim}
\subsection*{Quandles}
Let $Q$ be a quandle.
A \emph{subquandle} of $Q$ is a subset $R\subset Q$ 
that is closed under the quandle operation, namely, $x \ast y \in R$
for all $x,y\in R$.
For quandles $P,Q$, a map $f\colon P\to Q$ is called a \emph{quandle homomorphism}
(or a \emph{morphism of quandles}) if it satisfies
\[
f(x\ast y)=f(x)\ast f(y)\quad (x,y\in P).
\]
Let  $\As(Q)$ be the associated group of $Q$.
There is a well-defined right $\As(Q)$-action on $Q$ given by
\[
x\cdot e_y\coloneqq x\ast y\quad (x,y\in Q).
\]
This action satisfies
\[
g^{-1}e_xg=e_{x\cdot g}\in\As(Q)\quad (g\in\As(Q),x\in Q).
\]
An orbit of this action is sometimes
called a \emph{connected component} of $Q$.
In particular, if the action is transitive and hence $Q$ consists of a single orbit, then $Q$ is 
called \emph{connected}.
The following proposition is well-known:
\begin{pro}\label{pro:asq-abelianization}
Let $\mathcal{O}$ be a complete set of representatives of the orbits
of the right $\As(Q)$-action on $Q$.
Then the abelianization $\As(Q)_{\mathrm{Ab}}$ is a free abelian group with a basis
$\{[e_x]\mid x\in\mathcal{O}\}$.
\end{pro}

\subsection*{Symmetric quandles}
Let $(Q,\rho)$ be a symmetric quandle.
A symmetric quandle $(R,\tau)$ is called a \emph{symmetric subquandle} 
of $(Q,\rho)$ if $R$ is a subquandle of $Q$ and $\tau=\rho|_R$.
A \emph{symmetric quandle homomorphism} (or a \emph{morphism of symmetric quandles})
 $f\colon (P,\rho_{P})\to (Q,\rho_{Q})$
is a map $f\colon P\to Q$ satisfying
\[
f(x\ast y)=f(x)\ast f(y),\quad f(\rho_{P}(x))=\rho_{Q}(f(x)) \quad (x,y\in P).
\]
\begin{eg}
Let \( G \) be a group. As mentioned in \S\ref{sec:Intro}, 
the conjugation symmetric quandle \( (\Conj(G), \mathrm{Inv}(G)) \) is a fundamental example of a symmetric quandle. A subset \( Q \subset G \) is a symmetric subquandle of \( (\Conj(G), \mathrm{Inv}(G)) \) if and only if it is closed under both conjugation and inversion.
On the other hand, the \emph{core quandle} of \( G \), denoted by \( \mathrm{Core}(G) \), is the set $G$ equipped with the quandle operation
\[
g \ast h \coloneqq h g^{-1} h.
\]
The identity map is a good involution on \( \mathrm{Core}(G) \), and hence \( (\mathrm{Core}(G), \mathrm{id}_G) \) forms a symmetric quandle.
\end{eg}

\begin{eg}
A quandle \( Q \) is called \emph{involutive} (also referred to as \emph{involutory} or a \emph{kei} in the literature) if it satisfies \((x \ast y) \ast y = x\) for all \( x, y \in Q \).  
For instance, core quandles are involutive.
Involutive quandles form an important subclass of quandles~\cite{arXiv:1506.02389}.  
If \( Q \) is involutive, then the identity map \( \mathrm{id}_Q \) is a good involution of \( Q \), and hence \( (Q, \mathrm{id}_Q) \) is a symmetric quandle.  
Conversely, if \( (Q, \mathrm{id}_Q) \) is a symmetric quandle, then \( Q \) must be involutive~\cite{MR2657689}.
Involutive quandles are discussed in \S\ref{sec:involutive}.
\end{eg}

\begin{eg}
Let \( G \) be a group, and let \( \varphi \colon G \to G \) be an automorphism.  
The \emph{generalized Alexander quandle} \( \mathrm{Alex}(G, \varphi) \) associated with \( (G, \varphi) \) is the set \( G \) equipped with the quandle operation
\[
g \ast h \coloneqq \varphi(gh^{-1})h.
\]
If \( G \) is abelian, then \( \mathrm{Alex}(G, \varphi) \) is called the \emph{Alexander quandle} associated with \( (G, \varphi) \).  
It was proved in~\cite{MR4688855} that a generalized Alexander quandle admits a good involution if and only if it is involutive.
\end{eg}

\begin{eg}\label{eg:submfd}
Let $K$ be a (PL and locally flat or smooth) proper $n$-submanifold of an $(n+2)$-manifold $W$.  
The fundamental symmetric quandle $(Q(W, K), \rho)$ associated to the pair $(W, K)$ is defined in \cites{MR2371714,MR2657689}.  
If $W = \mathbb{R}^{n+2}$, then the associated group of the symmetric quandle $(Q(W, K), \rho)$  
is shown to be isomorphic to the fundamental group of the complement  
$\pi_1(\mathbb{R}^{n+2} \setminus K)$ in \cite{MR3363811}.
\end{eg}

\begin{eg}
Let \( Q \) be a quandle, and let \( \overline{Q} \) be a copy of \( Q \).  
One can define a quandle structure on the disjoint union  
\( D(Q) \coloneqq Q \sqcup \overline{Q} \), along with a good involution  
\( \rho \colon D(Q) \to D(Q) \), as follows.  
For each \( x \in Q \), denote by \( \bar{x} \in \overline{Q} \) the corresponding element.  
The quandle operation on \( D(Q) \) is defined by:
\[
\begin{aligned}
&x \ast y = x \ast y \in Q, 
&&x \ast \bar{y} = S_y^{-1}(x) \in Q, \\
&\bar{x} \ast y = \overline{x \ast y} \in \overline{Q}, 
&&\bar{x} \ast \bar{y} = \overline{S_y^{-1}(x)} \in \overline{Q}.
\end{aligned}
\]
The good involution \( \rho \colon D(Q) \to D(Q) \) is defined by  
\( \rho(x) = \bar{x} \) and \( \rho(\bar{x}) = x \).  
The symmetric quandle \( (D(Q), \rho) \) was introduced by Kamada~\cite{MR2371714},  
where it is referred to as the \emph{well-involuted double cover} of \( Q \).
In \cite{arXiv:2010.09983}, \( (D(Q), \rho) \) is called a \emph{symmetric double} of $Q$. 

\end{eg}

\section{Associated groups of quandles and central extensions}\label{sec:central-ext}
Now let $(Q,\rho)$ be a symmetric quandle and
$\As(Q,\rho)$ be the associated group.
Let $\pi_Q\colon \As(Q)\to\As(Q,\rho)$ be the canonical surjection defined by
$e_x\mapsto s_x$ $(x\in Q)$,
and denote the kernel of $\pi_Q$ by  $Z(Q,\rho)$.
Define $\sim$ as the equivalence relation on $Q$
generated by:
\begin{enumerate}
\item $x\sim x\ast y$\quad $(x,y\in Q)$,
\item $x\sim \rho(x)$\quad $(x\in Q)$.
\end{enumerate}
If $x,y\in Q$ belong to the same $\As(Q)$-orbit,
then $x\sim y$ by (1).
Let $\mathcal{C}=\{x_\lambda\mid\lambda\in\Lambda\}$ 
be a complete set of representatives of
the set of equivalence classes $Q/{\sim}$.
Now we can present the main result of this section:
\begin{thm}\label{thm:main}
The kernel $Z(Q,\rho)$ of the surjection
$\pi_Q\colon \As(Q) \to \As(Q,\rho)$ is central in $\As(Q)$, 
and it is a free abelian group 
with a basis $\{e_x e_{\rho(x)} \mid x \in \mathcal{C}\}$.
Consequently, there exists a
central extension of the form
\[
0\to\Z^{\oplus \mathcal{C}}\to\As(Q)\xrightarrow{\pi_Q}\As(Q,\rho)\to 1,
\]
where $\Z^{\oplus \mathcal{C}}$ is the free abelian group
with a basis $\mathcal{C}$.
Moreover, we have 
\[Z(Q,\rho) \cap [\As(Q), \As(Q)] = 1.\]
\end{thm}
The following corollary is an immediate consequence of the above theorem:
\begin{cor}\label{cor:commutator}
The canonical surjection $\pi_Q\colon\As(Q)\to\As(Q,\rho)$ 
induces an isomorphism
$[\As(Q),\As(Q)]\cong [\As(Q,\rho),\As(Q,\rho)]$.
\end{cor}
The proof of Theorem \ref{thm:main} is divided into several lemmas.
We begin with the following:
\begin{lem}\label{lem:commute}
Let $x\in Q$. Then $e_xe_{\rho(x)}=e_{\rho(x)}e_x$ in $\As(Q)$,
and the element $e_xe_{\rho(x)}$ belongs to the center of $\As(Q)$.
\end{lem}
\begin{proof} Observe that
\[
x\ast \rho(x)=S_x^{-1}(x)=x,
\]
and that
\[
e_{\rho(x)}^{-1}e_xe_{\rho(x)}=e_{x\ast\rho (x)}=e_x,
\]
from which the first assertion follows.
For the second assertion, consider any $y\in Q$. We compute:
\begin{align*}
(e_xe_{\rho(x)})^{-1}e_y(e_xe_{\rho(x)})&=
e_{\rho(x)}^{-1}(e_x^{-1}e_ye_x)e_{\rho(x)}\\
&=e_{\rho(x)}^{-1}e_{y\ast x}e_{\rho(x)}=e_{(y\ast x)\ast\rho(x)}=e_y.
\end{align*}
Hence, $e_x e_{\rho(x)}$ commutes with every generator $e_y$ for $y \in Q$, 
and thus lies in the center of $\As(Q)$.
\end{proof}

\begin{lem}\label{lem:central}
The kernel $Z {(Q,\rho)}$ of the canonical surjection 
$\pi_Q\colon\As(Q)\to\As(Q,\rho)$
is generated by $\{e_xe_{\rho(x)}\mid x\in Q\}$, and 
hence it is contained in the center of $\As(Q)$.
Consequently, the group extension
\[ 1\to Z {(Q,\rho)}\to\As(Q)
\to\As(Q,\rho)\to 1\]
is a central extension.
\end{lem}
\begin{proof}
By the definition of $\As(Q,\rho)$, 
the kernel $Z {(Q,\rho)}$ is the normal closure in $\As(Q)$ of the subset
\[\{e_x e_{\rho(x)}\mid x\in Q\}\subset\As(Q).\]
Since $g^{-1}(e_xe_{\rho(x)})g=e_xe_{\rho(x)}$ for all $g\in\As(Q)$
by Lemma \ref{lem:commute},
the kernel $Z {(Q,\rho)}$ is generated by $\{e_xe_{\rho(x)}\mid x\in Q\}$,
hence verifying the lemma.
\end{proof}

\begin{lem}\label{lem:comm-generators}
Let $x,y\in Q$.
If $x\sim y$ then $e_xe_{\rho(x)}=e_ye_{\rho(y)}\in Z {(Q,\rho)}$.
\end{lem}
\begin{proof}
In view of the definition of the equivalence relation $\sim$, it suffices to show that
$
e_{x \ast y} e_{\rho(x \ast y)} = e_x e_{\rho(x)} $
and
$
e_{\rho(x)} e_{\rho(\rho(x))} = e_x e_{\rho(x)}
$. 
We compute
\begin{align*}
e_{x\ast y}e_{\rho(x\ast y)}&=e_{x\ast y}e_{\rho(x)\ast y}=
(e_y^{-1}e_xe_y)(e_y^{-1}e_{\rho(x)}e_y) \\
&=e_y^{-1}e_xe_{\rho(x)}e_y=e_xe_{\rho(x)}
\end{align*}
and
\begin{align*}
&e_{\rho(x)}e_{\rho(\rho(x))}=e_{\rho(x)}e_x=e_xe_{\rho(x)}
\end{align*}
as desired.
\end{proof}

\begin{lem}
Let $\mathcal{C} = \{x_\lambda \mid \lambda \in \Lambda\}$ be as above, and let $O(x_\lambda)$ denote the $\As(Q)$-orbit of $x_\lambda \in \mathcal{C}$.
For each $\lambda \in \Lambda$, exactly one of the following two conditions holds:
\begin{enumerate}
\item $\rho(x_\lambda)\in O(x_\lambda)$,
\item $\rho(x_\lambda)\not\in O(x_\mu)$ for all 
$\mu\in\Lambda$. 
\end{enumerate}
\end{lem}
\begin{proof}
Suppose that $\rho(x_\lambda)\not\in O(x_\lambda)$ and that
$\rho(x_\lambda)\in O(x_\mu)$ for some $\mu\not=\lambda\in\Lambda$.
Then
$x_\lambda\sim\rho(x_\lambda)$ by the definition of 
the equivalence relation,
and $\rho(x_\lambda)\sim x_\mu$ because $\rho(x_\lambda)$ and $x_\mu$ belong
to the same orbit.
Hence $x_\lambda\sim x_\mu$,
but this contradicts the assumption that
$\mathcal{C}$ is a complete set of representatives.
\end{proof}
\begin{lem}\label{lem:orbits}
Define
\begin{align*}
\Lambda_1 & \coloneqq\{\lambda\in\Lambda
\mid\rho(x_\lambda)\in O(x_\lambda)\}, \\
\Lambda_2 & \coloneqq\{\lambda\in\Lambda
\mid\rho(x_\lambda)\not\in O(x_\mu)\ \text{{for all} $\mu\in\Lambda$}\}
\end{align*}
so that $\Lambda=\Lambda_1\sqcup\Lambda_2$.
Then all elements in
\[
\{{x_{\lambda}}\mid \lambda\in\Lambda\}\cup
 \{{\rho(x_{\mu})}\mid \mu\in\Lambda_2\}\subset Q
 \]
lie in mutually distinct $\As(Q)$-orbits. Consequently,
\[
\{[e_{x_{\lambda}}]\mid \lambda\in\Lambda\}\cup
 \{[e_{\rho(x_{\mu})}]\mid \mu\in\Lambda_2\}\subset\As(Q)_{\mathrm{Ab}}
 \]
is linearly independent in the free abelian group $\As(Q)_{\mathrm{Ab}}$.
\end{lem}
\begin{proof}
As for the first claim,
it suffices to prove that $O(\rho(x_\lambda))\not=O(\rho(x_\mu))$ 
whenever $\lambda\not=\mu$.
Indeed, if $O(\rho(x_\lambda))=O(\rho(x_\mu))$, then we have
\[x_\lambda\sim\rho(x_\lambda)\sim\rho(x_\mu)\sim x_\mu,\]
which implies $x_\lambda\sim x_\mu$ and hence $\lambda=\mu$.
This completes the proof of the first claim.
The second claim follows from Proposition \ref{pro:asq-abelianization}.
\end{proof}
\begin{rem}
Actually, the subset
$\{{x_{\lambda}}\mid \lambda\in\Lambda\}\cup
 \{{\rho(x_{\mu})}\mid \mu\in\Lambda_2\}\subset Q$ forms a complete set
 of representatives of $\As(Q)$-orbits of $Q$. See Proposition \ref{lem:embedd-orbit}.
\end{rem}

We are now ready to prove Theorem \ref{thm:main}.

\begin{proof}[Proof of Theorem \ref{thm:main}]
Since the subgroup $Z(Q,\rho)$ is central and is generated by
$\{e_xe_{\rho(x)}\mid x\in Q\}$ by Lemma \ref{lem:central},
it suffices to show that
\[\{e_xe_{\rho(x)}\mid x\in\mathcal{C}\}
=\{e_{x_\lambda}e_{\rho(x_\lambda)}\mid\lambda\in\Lambda\}\]
is linearly independent in $Z(Q,\rho)$.
Suppose that
\[
\prod_{\lambda\in\Lambda}(e_{x_\lambda}e_{\rho(x_\lambda)})^{c_\lambda}=1
\in Z(Q,\rho),
\]
where $c_\lambda\in\Z$ and $c_\lambda=0$ 
except finitely many $\lambda\in\Lambda$.
In $\As(Q)_{\mathrm{Ab}}$, we have
 \begin{align*}
 1=\prod_{\lambda\in\Lambda}[e_{x_\lambda}e_{\rho(x_\lambda)}]^{c_\lambda}
 =\prod_{\lambda\in\Lambda_1}[e_{x_\lambda}e_{\rho(x_\lambda)}]^{c_\lambda}
 \prod_{\lambda\in\Lambda_2}[e_{x_\lambda}e_{\rho(x_\lambda)}]^{c_\lambda}.
  \end{align*}
  If $\lambda\in\Lambda_1$ then $[e_{x_\lambda}]=[e_{\rho(x_\lambda)}]$,
  hence the right-hand side equals 
 \begin{align*}
 \prod_{\lambda\in\Lambda_1}[e_{x_\lambda}]^{2c_\lambda}
  \prod_{\lambda\in\Lambda_2}[e_{x_\lambda}]^{c_\lambda}
  \prod_{\lambda\in\Lambda_2}[e_{\rho(x_\lambda)}]^{c_\lambda}=1.
 \end{align*}
 
Since $\{[e_{x_{\lambda}}]\mid \lambda\in\Lambda\}\cup
 \{[e_{\rho(x_{\mu})}]\mid \mu\in\Lambda_2\}$ is linearly independent
 by Lemma \ref{lem:orbits},
 $c_\lambda=0$ for all $\lambda\in\Lambda$.
 This proves that $\{[e_xe_{\rho(x)}]\mid x\in\mathcal{C}\}$
 is linearly independent in $\As(Q)_{\mathrm{Ab}}$
 and $\{e_xe_{\rho(x)}\mid x\in\mathcal{C}\}$ is linearly 
 independent in $Z(Q,\rho)$.
Since $\{[e_xe_{\rho(x)}]\mid x\in\mathcal{C}\}$
 is linearly independent, 
the composition
\[Z_{(Q,\rho)}\hookrightarrow\As(Q)\twoheadrightarrow\As(Q)_{\mathrm{Ab}}\]
is injective, hence verifying $Z_{(Q,\rho)}\cap [\As(Q),\As(Q)]=1$.
\end{proof}

\section{Pullbacks}\label{sec:pullback}
In this section, we show that the associated groups $\As(Q)$ and $\As(Q,\rho)$ 
form a pullback diagram.
The proof is similar to the arguments in  \cites{MR4447657,hasegawa-thesis, MR3993765}.
First, we invoke the following result due to Kishimoto
\cite{MR3993765}:
\begin{pro}[Kishimoto \cite{MR3993765}] 
\label{pro:Kishimoto}
Suppose that there is a commutative square of groups
\[\begin{tikzcd}
G_1\arrow{r}{f_1} \arrow[d,"g"']& H_1\arrow{d}{h}\\
G_2\arrow{r}{f_2}&H_2
\end{tikzcd}\]
where $f_1$ is surjective.
Then the square is a pullback if and only if
the canonical map $\ker f_1\to\ker f_2$ is an isomorphism.
If the square is a pullback and $f_1,f_2$ are surjective, then the induced square
of Eilenberg-MacLane spaces
\[\begin{tikzcd}
K(G_1,1)\arrow{r} \arrow[d]& K(H_1,1)\arrow{d}\\
K(G_2,1)\arrow{r} &K(H_2,1)
\end{tikzcd}\]
is a homotopy pullback.
\end{pro}

Now consider the following commutative square:
\begin{equation}\label{eq:comm-square}
\begin{tikzcd}
\As(Q)\arrow{r}{\Ab_1} \arrow[d,"\pi_Q"']
& \As(Q)_{\Ab}\arrow{d}{(\pi_Q)_\Ab}\\
\As(Q,\rho)\arrow{r}{\Ab_2}& \As(Q,\rho)_{\Ab},
\end{tikzcd}\end{equation}
where $\pi_Q$ is the canonical surjection, $(\pi_Q)_\Ab$ is the homomorphism induced by $\pi_Q$,
and $\Ab_1,\Ab_2$ are the abelianization maps.
In view of Corollary \ref{cor:commutator}, the canonical homomorphism
$\ker \Ab_1=[\As(Q),\As(Q)]\to\ker \Ab_2=[\As(Q,\rho),\As(Q,\rho)]$ is an isomorphism.
Thus, Proposition~\ref{pro:Kishimoto} implies the following result:
\begin{thm}\label{thm:pullback}
For any symmetric quandle $(Q,\rho)$, the commutative square
\eqref{eq:comm-square}  is a pullback.
The induced square of Eilenberg-MacLane spaces
\[\begin{tikzcd}
K(\As(Q),1)\arrow{r} \arrow[d]
& K(\As(Q)_{\Ab},1)\arrow{d}\\
K(\As(Q,\rho),1)\arrow{r}& K(\As(Q,\rho)_{\Ab},1)
\end{tikzcd}\]
is a homotopy pullback.
\end{thm}
As an immediate corollary, we have the following:
\begin{cor}
Under the notations of \eqref{eq:comm-square}, for any symmetric quandle $(Q,\rho)$,
the homomorphism
\[
\As(Q)\to \As(Q,\rho)\times\As(Q)_\Ab,\quad g\mapsto (\pi_Q(g), [g])
\]
is injective.
\end{cor}

\section{Group-theoretic characterization of $\As(Q,\rho)$} 
\label{sec:group-theoretic-char}
Majid--Rietsch \cite{MR3065997} (see also Beggs--Majid \cite{MR4292536}*{\S1.7})
studied symmetric quandles under the name IP-quandles
(IP stands for ``inverse property'').
Their motivation came from the context of noncommutative differential geometry.
They introduced the notion of ``covering groups'', which we consider in this section.
As a byproduct, we obtain a group-theoretic characterization of associated groups
of symmetric quandles.

Let $G$ be a group.
Let $(Q,\rho)$ be a symmetric subquandle of
the conjugation symmetric quandle $(\Conj(G),\mathrm{Inv}(G))$.
Then the associated group $\As(Q,\rho)$ is given by
\begin{equation}\label{eq:As-of-conj}
\As(Q,\rho)=\langle \bare_g\ (g\in Q)\mid 
\bare_h^{-1}\bare_g \bare_h
=\bare_{h^{-1}gh}, \ \bare_{g^{-1}}=\bare_g^{-1}\ (g,h\in Q)\rangle,
\end{equation}
and there exists a canonical group homomorphism
$p_Q\colon\As(Q,\rho)\to G$ defined by $s_g\mapsto g$ $(g\in G)$,
which is surjective if $Q$ generates $G$.
Under the assumption that $Q$ generates $G$ and $Q\subset G\setminus\{1\}$,
Majid--Rietsch \cite{MR3065997} proved that the kernel of $p_Q$ is
a central subgroup of $\As(Q,\rho)$ so that
\[
1\to\ker p_{Q}\to \As(Q,\rho)\xrightarrow{p_Q} G\to 1
\]
is a central extension.
A group $G$ is called a \emph{covering group}, as defined by
Majid--Rietsch \cite{MR3065997},
if there exists a symmetric subquandle $(Q,\rho)$
of $(\Conj(G),\mathrm{Inv}(G))$
such that
\begin{enumerate}
\item $Q\subset G\setminus\{1\}$,
\item $Q$ generates $G$,
\item $p_Q\colon\As(Q,\rho)\to G$ is an isomorphism.
\end{enumerate}
Majid and Rietsch \cite{MR3065997} proved that every finite crystallographic reflection group $W$ is a covering group with respect to $(Q, \mathrm{id}_Q)$, where $Q$ is the set of reflections in $W$. 
We generalize their result through the notion of twisted Wirtinger presentations, whose definition we now recall:
\begin{defn}
We say that a group presentation $\langle X\mid R\rangle$ is a 
\emph{twisted Wirtinger presentation} if
each relation in $R$ is of the form 
\[
w^{-1}xw=y^\epsilon\quad (x,y\in X, w\in F(X),\epsilon\in\{\pm 1\}).
\]
In particular, if all the relations in $R$ are of the form
\[
w^{-1}xw=y\quad (x,y\in X, w\in F(X)),
\]
then $\langle X\mid R\rangle$ is called a \emph{Wirtinger presentation}.
\end{defn}

\begin{thm}\label{thm:C-group}
If a group $G$ admits a twisted Wirtinger presentation,
then the group $G$ is a covering group.
\end{thm}

\begin{proof}
Let $\langle X\mid R\rangle$ be a twisted Wirtinger presentation of $G$.
Define
\[X^{-}\coloneqq\{x^{-1}\mid x\in X\},\quad
X^{\pm}\coloneqq X\cup X^-,\]
and set
\[
Q\coloneqq \bigcup_{g\in G}g^{-1}X^\pm g.
\]
Let $\rho\colon Q\to Q$ be the involution defined by $g\mapsto g^{-1}$.
Then $(Q,\rho)$ is a symmetric subquandle of $(\Conj(G),\mathrm{Inv}(G))$
satisfying conditions (1) and (2).
We will prove that the surjection $p_Q\colon\As(Q,\rho)\twoheadrightarrow G$
is an isomorphism.
First, we claim that
$\As(Q,\rho)$ is generated by $s_x$ $(x\in X)$.
Given $x\in Q$, we can express $x$ as
\[
x=(x_1x_2\cdots x_k)^{-1}x_0(x_1x_2\cdots x_k)\quad (x_i\in X^\pm).
\]
By the iterated use of relations
$\bare_y^{-1}\bare_x\bare_y=\bare_{x\ast y}=
\bare_{y^{-1}x {y^{}}}$ $(x,y\in Q)$, we have
\[
\bare_x=\bare_{(x_1x_2\cdots x_k)^{-1}x_0(x_1x_2\cdots x_k)}
=(\bare_{x_1}\bare_{x_2}\cdots \bare_{x_k})^{-1}
\bare_{x_0}(\bare_{x_1}\bare_{x_2}\cdots \bare_{x_k}).
\]
If $x_i\in X^-$ then $x_i^{-1}\in X$ and
$\bare_{x_i}=(\bare_{{x_i}^{-1}})^{-1}$, proving that
$\As(Q,\rho)$ is generated by $s_x$ $(x\in X)$.
Now we prove the theorem.
For each relation 
\[
w^{-1}xw=y^\epsilon
\]
in $R$, 
the element $w\in F(X)$ is expressed as a word on $X^\pm$, say
\[w=x_1x_2\cdots x_k\quad (x_i\in X^\pm),\]
then the relation is rephrased by
\begin{equation}\label{eq:C-gp-rel}
(x_1x_2\cdots x_k)^{-1}x(x_1x_2\cdots x_k)=y^\epsilon.
\end{equation}
By the iterated use of relations
$\bare_y^{-1}\bare_x\bare_y
=\bare_{y^{-1}x {y^{}}}$
$(x,y\in Q)$ again,
it follows that
\[
(\bare_{x_1}\cdots \bare_{x_k})^{-1}\bare_x
(\bare_{x_1}\cdots \bare_{x_k})
=\bare_{(x_1\cdots x_k)^{-1}x
(x_1\cdots x_k)}
=\bare_{y^\epsilon}=(\bare_y)^\epsilon 
\]
holds in $\As(Q,\rho)$, which means that the relation 
\eqref{eq:C-gp-rel} is valid in $\As(Q,\rho)$.
Since $\As(Q,\rho)$ is generated by $s_x$ $(x\in X)$ and
the homomorphism $p_Q$ is surjective, $p_Q$ must be an isomorphism.
\end{proof}
As mentioned in the introduction, Coxeter groups admit twisted Wirtinger presentation.
Since finite crystallographic reflection groups are Coxeter groups, 
Theorem \ref{thm:C-group} generalizes a
result of Majid--Rietsch \cite{MR4292536} mentioned above.
As a byproduct of Theorem \ref{thm:C-group}, we obtain the 
following corollary:
\begin{cor}\label{cor:char-As}
For a group $G$, the following two conditions are equivalent:
\begin{enumerate}
\item $G$ is isomorphic to $\As(Q,\rho)$ for some symmetric quandle $(Q,\rho)$.
\item $G$ admits a twisted Wirtinger presentation.
\end{enumerate}
\end{cor}
\begin{proof}
Condition (2) implies (1) by Theorem \ref{thm:C-group}. Conversely, 
for any symmetric quandle $(Q,\rho)$, the presentation of its associated group
\[
\As(Q,\rho)\coloneqq\langle \bare_x\ (x\in Q)\mid \bare_y^{-1}\bare_x \bare_y
=\bare_{x\ast y}, \ \bare_{\rho(x)}=\bare_x^{-1}\ (x,y\in Q)\rangle
\]
can be transformed into a twisted Wirtinger presentation.
Indeed, the relation $\bare_{\rho(x)}=\bare_x^{-1}$ is equivalent to
the relation ${\bare_{\rho(x)}}^{-1}\bare_{\rho(x)}\bare_{\rho(x)}=\bare_x^{-1}$.
\end{proof}

Groups admitting Wirtinger presentations are known by various names:
they are called \emph{Wirtinger groups} in \cite{Yajima},
\emph{labelled oriented graph groups}
(\emph{LOG groups} in short) in \cite{Howie},
and \emph{C-groups} 
in \cites{Kulikov,MR1307063}.
Groups that admit twisted Wirtinger presentations are referred to as \emph{C-groups} 
in Kuz\cprime min \cite{MR1392843}, 
who also provides a characterization of these groups in terms of their second group homology.
Building on his result, the first author and Takase~\cite{MR4849109} 
further studied the second integral homology of groups admitting Wirtinger presentations.

\section{Abelianization of associated groups of symmetric quandles}\label{sec:abelianization}
Let $(Q,\rho)$ be a symmetric quandle.
In this section, we determine the abelianization $\As(Q,\rho)_{\mathrm{Ab}}$
of the associated group $\As(Q,\rho)$, and prove that $\As(Q,\rho)_{\mathrm{Ab}}$ is isomorphic
to the first symmetric quandle homology of $(Q,\rho)$.
Recall that the presentation of the associated group is given by
\[
\As(Q,\rho) = \langle s_x \ (x\in Q) \mid s_y^{-1} s_x s_y = s_{x \ast y},\  
s_{\rho(x)} = s_x^{-1}\ (x,y\in Q) \rangle.
\]
As a consequence, we obtain the presentation of the abelianization $\As(Q,\rho)_{\mathrm{Ab}}$,
namely it is generated by
$[s_x] \ (x\in Q)$ subject to the relations
\begin{equation}\label{eq:presen-abelianization}
[s_x] = [s_{x \ast y}] , \quad [s_{\rho(x)}] = [s_x]^{-1} , \quad [s_x] [s_y] = [s_y] [s_x] 
\quad (x,y\in Q).
\end{equation}
In particular, if $x,y\in Q$ belong to the same $\As(Q)$-orbit, then $[s_x] = [s_y]$.
According to Proposition \ref{lem:embedd-orbit},
exactly one of the following three conditions holds:
\begin{enumerate}
\item $x\in O(x_\lambda)$ for some $\lambda \in \Lambda_1$,  
\item $x\in O(x_\lambda)$ for some $\lambda \in \Lambda_2$,  
\item $x\in O(\rho(x_\lambda))$ for some $\lambda \in \Lambda_2$,
\end{enumerate}
where $\mathcal{C}=\{x_\lambda\mid \lambda\in\Lambda\}$ is a complete set of representatives
of $Q/{\sim}$ and
\begin{align*}
&\Lambda_1\coloneqq\{\lambda\in\Lambda
\mid\rho(x_\lambda)\in O(x_\lambda)\},\\
&\Lambda_2\coloneqq\{\lambda\in\Lambda
\mid\rho(x_\lambda)\not\in O(x_\mu)\ \text{for all $\mu\in\Lambda$}\}.
\end{align*}
As a consequence, $\As(Q,\rho)_{\mathrm{Ab}}$ is generated by
\[
\{ [s_{x_{\lambda}}] \mid \lambda \in \Lambda_1 \} \cup 
\{ [s_{x_{\lambda}}] \mid \lambda \in \Lambda_2 \} \cup 
\{ [s_{\rho(x_{\lambda})}] \mid \lambda \in \Lambda_2 \}.
\]
If $\lambda \in \Lambda_1$ then $[s_{\rho(x_{\lambda})}] = [s_{x_{\lambda}}]$ and 
$[s_{x_{\lambda}}]^2 =[s_{x_{\lambda}}][s_{\rho(x_{\lambda})}] = 1$.
If $\lambda \in \Lambda_2$ then $[s_{\rho(x_{\lambda})}] = [s_{x_{\lambda}}]^{-1}$.
We denote $[s_\lambda]$ instead of $[s_{x_\lambda}]$ for simplicity,
and with this notation, we can deduce the following theorem:
\begin{thm}\label{thm:abelianization}
The abelianization $As(Q,\rho)_{\mathrm{Ab}}$ is generated by
$[s_{\lambda}]$ $(\lambda\in\Lambda)$ subject to the relations
\[
[s_\lambda][s_\mu]=[s_\mu][s_\lambda]\ (\lambda,\mu\in\Lambda),\quad
[s_\lambda]^2=1\ (\lambda\in\Lambda_1).\]
Consequently, we have 
\[
\As(Q,\rho)_{\mathrm{Ab}}=
\bigoplus_{\lambda \in \Lambda_1} \mathbb{Z}/2\mathbb{Z} [s_{{\lambda}}]
\oplus 
\bigoplus_{\lambda \in \Lambda_2} \mathbb{Z} [s_{{\lambda}}]
\cong(\Z/2\Z)^{\oplus\Lambda_1}\oplus \Z^{\oplus\Lambda_2}.
\]
\end{thm}

The symmetric quandle homology $H_*(Q,\rho)$ of a symmetric quandle $(Q,\rho)$
was introduced by Kamada \cite{MR2371714} (see also Kamada--Oshiro \cite{MR2657689}).
We now show that the first homology group 
$H_1(Q,\rho)$ is isomorphic to the abelianization of the associated group $\As(Q,\rho)$.
Although this result is almost certainly known to experts,
we could not find it explicitly stated in the literature.

Let $(Q,\rho)$ be a symmetric quandle.
Define $C_n(Q)$ to be the free abelian group with a basis $Q^n$ for $n=1,2$.
We denote a basis element of $C_1(Q)$ by $(x)\in C_1(Q)$, 
and a basis element of  $C_2(Q)$ by $(x,y)\in C_2(Q)$.
Define the abelian group $C_1(Q,\rho)$ to be the quotient of $C_1(Q)$ by
the subgroup generated by 
\[\{(x)+(\rho(x))\mid x\in Q\}.\]
 Let $\partial\colon C_2(Q)\to C_1(Q,\rho)$ be  the homomorphism defined by
\[
\partial(x,y)\coloneqq (x)-(x\ast y)\quad (x,y\in Q).
\]
Then the first symmetric quandle homology is defined by 
\[H_1(Q,\rho)\coloneqq C_1(Q,\rho)/\Image\,\partial.\]
In other words, $H_1(Q,\rho)$ is the abelian group generated by symbols $[x]$ $(x\in Q)$
subject to the relations
\[
[x]=[x\ast y], \quad [\rho(x)]=-[x]\quad (x,y\in Q),
\]
which coincides with the presentation of $\As(Q,\rho)_\Ab$ given in 
equation \eqref{eq:presen-abelianization}.
Thus we have proved:
\begin{pro}\label{cor:1st-symm-qdl-homology}
Let $(Q,\rho)$ be a symmetric quandle, then $H_1(Q,\rho)\cong\As(Q,\rho)_\Ab$.
\end{pro}

\section{Left adjointness}\label{sec:adjoint}
Let $\Gp$ denote the category of groups and group homomorphisms,
$\Qdl$ the category of quandles and quandle homomorphisms,
and $\SQ$ the category of symmetric quandles and symmetric quandle homomorphisms.
As mentioned in the introduction, 
the associated group functor $\Qdl \to \Gp$,
which assigns the associated group $\As(Q)$ 
to a quandle $Q$,  is left adjoint to
the conjugation quandle functor $\Gp \to \Qdl$,
which assigns the conjugation quandle $\Conj(G)$ 
to a group $G$.

In this section we prove that a similar adjoint relationship holds for the category 
$\SQ$ of symmetric quandles.
Let $\As\colon \SQ \to \Gp$ be the functor that assigns the associated group 
$\As(Q,\rho)$ to a symmetric quandle $(Q,\rho)$.
Let $\Conj\colon \Gp \to \SQ$ be the functor that assigns 
the conjugation symmetric quandle $(\Conj(G),\Inv(G))$ to a group $G$.
 We aim to prove that the functor $\As$ is left adjoint to $\Conj$.
 To this end, we invoke the following proposition:
 \begin{pro}[Kamada--Oshiro \cite{MR2657689}]\label{pro:universality-KO}
Let $(Q,\rho)$ be a symmetric quandle, and let $G$ be a group.
The natural map $\mu\colon Q\to\As(Q,\rho)$ is a morphism of symmetric quandles:
\[
\mu\colon (Q,\rho)\to (\Conj(\As(Q),\rho),\Inv(\As(Q,\rho))).
\]
For any morphism of symmetric quandles $f\colon (Q,\rho)\to (\Conj(G),\Inv(G))$,
there exists a unique group homomorphism $\theta(f)\colon\As(Q,\rho)\to G$ 
with $f=\theta(f)\circ\mu$.
That is, the following diagram commutes:
\[\begin{tikzcd}
(Q,\rho) \arrow[r, "\mu"] \arrow[d, "f"'] & \As(Q,\rho) \arrow[d, "\theta(f)"] \\
(\Conj(G),\Inv(G)) \arrow[r, "\mathrm{id}"]         & G                  
\end{tikzcd}\]
 \end{pro}

\begin{pro}\label{pro:adjoint}
The functor $\As$ is left adjoint to $\Conj$.
\end{pro}
\begin{proof}
To prove that $\As$ is left adjoint to $\Conj$, 
we must show that for any object $(Q,\rho)$ in $\SQ$ and any group $G$ in $\Gp$, 
there exists a natural bijection
\[
\Hom_{\SQ}((Q,\rho), (\Conj(G),\Inv(G)) \cong \Hom_{\Gp}(\As(Q,\rho), G).
\]
Proposition~\ref{pro:universality-KO} provides the map
\[
\theta \colon \operatorname{Hom}_{\SQ}\left( (Q, \rho), \left( \operatorname{Conj}(G), \operatorname{Inv}(G) \right) \right)
\to
\operatorname{Hom}_{\Gp}\left( \operatorname{As}(Q, \rho), G \right)
\]
by assigning to each symmetric quandle homomorphism \(f\) the unique group homomorphism \(\theta(f)\) satisfying \(f = \theta(f) \circ \mu\).
Conversely, we define the map
\[
\eta \colon \operatorname{Hom}_{\Gp}(\operatorname{As}(Q, \rho), G) 
\to \operatorname{Hom}_{\SQ}\left((Q, \rho), \left(\operatorname{Conj}(G), \operatorname{Inv}(G)\right)\right)
\]
by the rule \(\eta(g)(q) = g(s_q)\), for each \(g \in \operatorname{Hom}_{\Gp}(\operatorname{As}(Q, \rho), G)\) and \(q \in Q\),  
where \(s_q\) denotes the canonical generator of \(\operatorname{As}(Q, \rho)\) corresponding to the element \(q \in Q\).

We now verify that the maps $\theta$ and $\eta$ are inverses each other.
Indeed, for any 
$f \in \Hom_{\SQ}((Q,\rho),(\Conj(G),\Inv(G)))$, we have
\[
( \eta \circ \theta (f) ) (q) = \theta (f) (s_q) = f(q),
\]
and for any $g \in \Hom_{\Gp}(\As(Q,\rho),G)$, we have
\[
( \theta \circ \eta (g) ) (s_q) = \eta (g) (q) = g(s_q).
\]
Thus $\eta \circ \theta$ and $\theta \circ \eta$ are identity maps, and two maps $\theta$ and $\eta$ are
mutually inverse.

To prove naturality, we must show that for any morphisms 
\[
\alpha \in \operatorname{Hom}_{\SQ}((P, \sigma), (Q, \rho))\quad \text{and}\quad
\beta \in \operatorname{Hom}_{\Gp}(G, H), 
\]
the following diagram commutes:
\[
\begin{tikzcd}
\operatorname{Hom}_{\SQ}\left((Q, \rho), (\operatorname{Conj}(G), \operatorname{Inv}(G))\right)
\arrow[r, "\theta"] \arrow[d, "(\operatorname{Conj}(\beta) \circ - \circ \alpha)"'] 
& \operatorname{Hom}_{\Gp}(\operatorname{As}(Q, \rho), G) \arrow[d, "\beta \circ - \circ \operatorname{As}(\alpha)"] \\
\operatorname{Hom}_{\SQ}\left((P, \sigma), (\operatorname{Conj}(H), \operatorname{Inv}(H))\right)
\arrow[r, "\theta"'] & \operatorname{Hom}_{\Gp}(\operatorname{As}(P, \sigma), H)
\end{tikzcd}
\]
To verify this, let 
\(f \in \operatorname{Hom}_{\SQ}\left((Q, \rho), (\operatorname{Conj}(G), \operatorname{Inv}(G))\right)\),  
and let \(s_{p} \in \operatorname{As}(P, \sigma)\) denote the canonical generator corresponding to \(p \in P\). Then:
\[
(\beta \circ \theta(f) \circ \operatorname{As}(\alpha))(s_{p}) 
= (\beta \circ \theta(f))(s_{\alpha(p)}) 
= (\beta \circ f \circ \alpha)(p).
\]
On the other hand,
\[
\theta(\operatorname{Conj}(\beta) \circ f \circ \alpha)(s_{p}) 
= (\operatorname{Conj}(\beta) \circ f \circ \alpha)(p) 
= (\beta \circ f \circ \alpha)(p).
\]
Since both expressions agree for all generators \(s_{p}\), the two compositions coincide.  
Therefore, the bijection \(\theta\) is natural in both variables, completing the proof.
\end{proof}

\section{Second quandle homology}\label{sec:2nd-qdl-homology}
In this section, we apply Corollary \ref{cor:commutator} to the second quandle
homology of quandles. 
The proof is similar to the arguments in \cites{MR4447657}.
Let $(Q,\rho)$ be a symmetric quandle such that $Q$ is a \emph{connected} quandle.
We consider the second quandle homology $H_{2}(Q)$ of $Q$.
For the definition of quandle homology, we refer to \cite{MR2371714,MR2657689,nosaka-book}.
Choose $x_{0}\in Q$ and let $\mathrm{Stab}_{\As(Q)}(x_{0})$ be the
isotropy subgroup at $x_{0}$ of the $\As(Q)$-action on $Q$.
In \cite{MR3205568}, Eisermann showed that the second quandle homology group $H_2(Q)$ 
is isomorphic to
\[
H_{2}(Q)\cong (\mathrm{Stab}_{\As(Q)}(x_{0})\cap [\As(Q),\As(Q)])_{\mathrm{Ab}}.
\]
Recall that $Z(Q,\rho)$ is the kernel of the canonical surjection 
$\pi_Q\colon\As(Q)\to\As(Q,\rho)$.
\begin{lem}\label{lem:action}
$Z(Q,\rho)$ acts trivially on $Q$, and hence the $\As(Q)$-action on $Q$
induces the well-defined $\As(Q,\rho)$-action on $Q$
given by $x\cdot s_{y}\coloneqq x\ast y$ $(x,y\in Q)$.
\end{lem}
\begin{proof}
To verify the lemma, it suffices to show that
$x\cdot e_{y}e_{\rho(y)}=x$ for all $x,y\in Q$,
because $\Z(Q,\rho)$ is generated by $\{e_{y}e_{\rho(y)}\mid y\in Q\}$ by Lemma \ref{lem:central}.
We have 
\begin{align*}
x\cdot e_{y}e_{\rho(y)}=(x\ast y)\ast \rho(y)=S_{y}^{-1}(S_{y}(x))=x
\end{align*}
as desired.
\end{proof}
Let $\mathrm{Stab}_{\As(Q,\rho)}(x_{0})$ be the
isotropy subgroup at $x_{0}$ of the $\As(Q,\rho)$-action on $Q$.
In view of Lemma \ref{lem:action}, $\As(Q)$ acts on $Q$
through the canonical surjection 
$\pi_Q\colon\As(Q)\to\As(Q,\rho)$.
Hence the canonical surjection induces a surjective homomorphism
$\mathrm{Stab}_{\As(Q)}(x_{0})\to \mathrm{Stab}_{\As(Q,\rho)}(x_{0})$.
Together with Corollary \ref{cor:commutator}, we conclude that
$\pi_Q\colon\As(Q)\to\As(Q,\rho)$ induces an isomorphism
\[
\mathrm{Stab}_{\As(Q)}(x_{0})\cap [\As(Q),\As(Q)]\cong
\mathrm{Stab}_{\As(Q,\rho)}(x_{0})\cap [\As(Q,\rho),\As(Q,\rho)],
\]
and hence we obtain the following result:
\begin{cor}\label{cor:homology}
Let $(Q,\rho)$ be a symmetric quandle whose underlying quandle $Q$ is connected,
and let $x_{0}\in Q$. We have
\[
H_{2}(Q)\cong (\mathrm{Stab}_{\As(Q,\rho)}(x_{0})\cap [\As(Q,\rho),\As(Q,\rho)])_{\mathrm{Ab}}.
\]
\end{cor}
\begin{eg}\label{eg:2nd-homology-conj}
Let $G$ be a group admitting a twisted Wirtinger presentation $\langle X\mid R\rangle$,
and let $(Q,\rho)$ be the symmetric quandle defined by
\[
Q\coloneqq \bigcup_{g\in G}g^{-1}X^\pm g,\quad\rho\colon g\mapsto g^{-1}
\]
as in the proof of Theorem \ref{thm:C-group}.
In that proof, it was shown that the homomorphism $p_Q\colon\As(Q,\rho)\to G$,
defined by $s_g\mapsto g$, is an isomorphism.
Identifying $\As(Q,\rho)$ with $G$ via $p_Q$, we obtain
\[
\mathrm{Stab}_{\As(Q,\rho)}(x_{0})=C_G(x_0),
\]
where $C_G(x_0)$ denotes the centralizer of $x_0$ in $G$. 
It follows from Corollary \ref{cor:homology} that if 
$Q$ is a \emph{connected} quandle, then
\begin{equation}\label{eq:H2Q}
H_{2}(Q)\cong (C_G(x_0)\cap [G,G])_\Ab.
\end{equation}
\end{eg}

\section{Involutive quandles}\label{sec:involutive}
Our results apply particularly well to involutive quandles.  
Recall that a quandle \( Q \) is involutive if \((x \ast y) \ast y = x\) for all \( x, y \in Q \).  
For an involutive quandle \( Q \), the identity map \( \mathrm{id}_Q \) is a good involution,  
and hence \((Q, \mathrm{id}_Q)\) forms a symmetric quandle.
For such a symmetric quandle \((Q, \mathrm{id}_Q)\), the associated group is given by
\[\As(Q,\mathrm{id}_Q)=\langle \bare_x\ (x\in Q)\mid 
\bare_y^{-1}\bare_x \bare_y
=\bare_{x\ast y}, \ \bare_x^{2}=1\ (x,y\in Q)\rangle.
\] 

Observe that
$x\sim y$ if and only if $x$ and $y$ belong to the same $\As(Q)$-orbit.
Let $\mathcal{O}$ be a complete set of representatives of orbits.
As an immediate consequence of Theorem \ref{thm:main} and Theorem \ref{thm:abelianization}, 
we obtain the following corollary:
\begin{cor}\label{cor:involutive}
Let $Q$ be an involutive quandle, let $\pi_Q\colon\As(Q)\to\As(Q,\mathrm{id}_Q)$
be the canonical surjection, and let $Z {(Q,\mathrm{id}_Q)}$ be its kernel.
Then $Z {(Q,\mathrm{id}_Q)}$ is a free abelian group with basis
$\{e_{x}^{2}\mid x\in\mathcal{O}\}$,  and there is a central extension of the form
\[
0\to\Z^{\oplus\mathcal{O}}\to \As(Q)\xrightarrow{\pi_Q}\As(Q,\mathrm{id}_Q)\to 1.
\]
For the abelianization $\As(Q,\mathrm{id}_Q)_{\Ab}$, we have
\[
\As(Q,\mathrm{id}_Q)_{\Ab}=\bigoplus_{x\in\mathcal{O}}\Z/2\Z[s_{x}]\cong (\Z/2\Z)^{\oplus\mathcal{O}}.
\]
\end{cor}
It is also worth noting that, when $Q$ is a finite involutive quandle, $\As(Q,\mathrm{id}_Q)$ coincides
with a ``finite quotient of the structure group'' studied in Lebed--Vendramin \cite{MR3974961},
where it is denoted by $\bar{G}_{(X,\triangleleft)}$.
The following theorem is a consequence of their results:
\begin{thm}[Lebed--Vendramin \cite{MR3974961}]
Let $Q$ be a finite involutive quandle. Then $\As(Q,\mathrm{id}_Q)$ is a finite group.
\end{thm}
\begin{rem}
Let $(Q,\rho)$ be a symmetric quandle. 
Even if $Q$ is finite, the associated group $\As(Q,\rho)$ may be infinite.
A simple example was given by Kamada--Oshiro~\cite{MR2657689}.
Specifically,
let $Q=\{x,y\}$ be a trivial quandle, and let $\rho\colon Q\to Q$
be the transposition that swaps $x$ and $y$.
Then $(Q,\rho)$ is a symmetric quandle, and $\As(Q,\rho)\cong\Z$.
\end{rem}

\begin{eg}
Let us consider Example \ref{eg:2nd-homology-conj} in the case of involutive quandles.
Let $\langle X\mid R\rangle$ be a Wirtinger presentation.
By adding relations $x^2=1$ for all $x\in X$, we obtain a
twisted Wirtinger presentation 
\[
G\coloneqq\langle X\mid R\sqcup\{x^2\mid x\in X\}\rangle,
\]
and the corresponding involutive quandle
\[
Q\coloneqq \bigcup_{g\in G}g^{-1}Xg.
\]
The quandle $Q$ is connected if and only if all generators $x\in X$ are
mutually conjugate in $G$, and in that case, for any $x_0\in X$, we have
the isomorphism as in \eqref{eq:H2Q}:
\[
H_2(Q)\cong (C_G(x_0)\cap [G,G])_\Ab
\]

A particular example is the Coxeter quandle $Q_W$ associated with a Coxeter system $(W,S)$.
For background on Coxeter groups and systems, see Humphreys~\cite{humphreys}.
Let $(W,S)$ be a Coxeter system, and define
\[Q_W\coloneqq\bigcup_{w\in W}w^{-1}Sw,\]
the set of reflections of $W$. 
Then $Q_W$ forms an involutive quandle with respect to the quandle operation given by
\[x\ast y=y^{-1}xy=yxy\quad (x,y\in Q_W).\]
This quandle $Q_W$ is called the \emph{Coxeter quandle} associated with $(W,S)$
in \cites{MR4175808,hasegawa-thesis,MR3821082,nosaka-book,arXiv:2506.23175,MR4669143}.
Note that some people use the term ``Coxeter quandle'' to refer to a different kind of quandle.
If all elements in $S$ are mutually conjugate in $W$, for any $s_0\in S$, we have an isomorphism
\begin{equation}\label{eq:homology-Coxeter}
H_2(Q_W)\cong (C_W(s_0)\cap [W,W])_\Ab.
\end{equation}
\end{eg}

\begin{rem}
The associated group of a Coxeter quandle was studied in detail by
the first author \cite{MR4175808}.
Hasegawa \cite{hasegawa-thesis} investigated the associated groups 
$\As(Q)$ and $\As(Q, \mathrm{id}_Q)$ of involutive quandles $Q$, 
without being aware of the notion of symmetric quandles.
 In particular, he proved Theorem \ref{thm:main}, Theorem \ref{thm:pullback}, and 
 Corollary \ref{cor:commutator} in the setting of  involutive quandles.
Using the isomorphism \eqref{eq:homology-Coxeter}, he also computed the second quandle homology
of all finite connected Coxeter quandles in \cite{hasegawa-thesis}.
\end{rem}

\section{Embeddability}\label{sec:embedd}
A quandle $Q$ is called \emph{embeddable} 
in $\As(Q)$ (also called \emph{injective} in the literature) if the map
$Q\to\As(Q)$ defined by $x\mapsto e_x$ $(x\in Q)$ is injective.
If $Q$ is embeddable, then $Q$ is isomorphic to
a subquandle of the conjugation quandle $\Conj(\As(Q))$.
The embeddability of finite quandles appears to be relevant 
in the study of set-theoretic solutions to the Yang-Baxter equation 
(see \cites{arXiv:2506.23175,MR3974961}
for instance).
One of the most well-known results concerning embeddability is a theorem of
Ryder \cite{MR1388194}, who showed that
the fundamental quandle of a classical knot in $S^3$ is embeddable
if and only if the knot is prime.
For a compiled list of known embeddable quandles, we refer the reader to the introduction of \cite{MR4564617}.

Similarly, we say that a symmetric quandle $(Q, \rho)$ is \emph{embeddable} in $\As(Q,\rho)$ if  
the map $Q \to \As(Q, \rho)$ defined by $x \mapsto s_x$ $(x \in Q)$ is injective.
Observe that if a symmetric quandle $(Q, \rho)$ is embeddable,  
then the underlying quandle $Q$ is also embeddable.  
Indeed, the map $Q \to \As(Q, \rho)$ factors through $\As(Q)$ as the composition  
\[
Q \to \As(Q) \twoheadrightarrow \As(Q, \rho).
\]
We show that the converse holds.
To this end, we use the following proposition:
\begin{pro}\label{lem:embedd-orbit}
Let $\Lambda_1,\Lambda_2$, and $\Lambda$ be as in Lemma \ref{lem:orbits}.
For any $x\in Q$, exactly one of the following three conditions holds:
\begin{enumerate}
\item $x\in O(x_\lambda)$ for some $\lambda \in \Lambda_1$,  
\item $x\in O(x_\lambda)$ for some $\lambda \in \Lambda_2$,  
\item $x\in O(\rho(x_\lambda))$ for some $\lambda \in \Lambda_2$,
\end{enumerate}
where $O(y)$ denotes the $\As(Q)$-orbit of $y\in Q$.
Consequently, the subset
\[
\{x_{\lambda}\mid\lambda\in\Lambda\}\cup 
\{\rho(x_{\lambda})\mid\lambda\in\Lambda_{2}\}\subset Q
\]
forms a complete set of representatives of $\As(Q)$-orbits of $Q$.
\end{pro}
\begin{proof}
Recall that the right $\As(Q)$-action on $Q$ is defined by $x\cdot e_{y}=x\ast y$
$(x,y\in Q)$.
We have $\rho(x\ast y)=\rho(x)\ast y$
and hence 
\begin{equation}\label{eq:equivariant1}
\rho(x\cdot e_y)=\rho(x)\cdot e_y\quad (x,y\in Q).
\end{equation}
We claim that 
\begin{equation}\label{eq:equivariant2}
\rho(x\cdot e_y^{-1})=\rho(x)\cdot e_y^{-1}\quad (x,y\in Q).
\end{equation}
This follows from the computation:
\begin{align*}
&\rho(x\cdot e_y^{-1})=\rho(S_{y}^{-1}(x))=\rho(x\ast\rho(y))=\rho(x)\ast\rho(y),\\
&\rho(x)\cdot e_y^{-1}=S_{y}^{-1}(\rho(x))=\rho(x)\ast\rho(y).
\end{align*}
As a consequence of \eqref{eq:equivariant1} and \eqref{eq:equivariant2}, we obtain
\begin{equation}\label{eq:equivariant}
\rho(x\cdot g)=\rho(x)\cdot g\quad (x\in Q,g\in\As(Q)),
\end{equation}
that is, the involution $\rho$ is $\As(Q)$-equivariant.
Given $x\in Q$, there exists a unique $\lambda\in\Lambda=\Lambda_1\sqcup\Lambda_2$ 
such that $x\sim x_\lambda$.
By the definition of the equivalence relation $\sim$, 
the element $x_\lambda$ can be transformed into $x$ through 
repeated applications of the following two types of operations and their inverses:
\begin{enumerate}
\item $y\mapsto y\ast z=y\cdot e_z\quad (y,z\in Q)$,
\item $y\mapsto \rho(y)\quad (y\in Q)$.
\end{enumerate}
In view of  \eqref{eq:equivariant}, we conclude that $x=x_\lambda\cdot g$
 or $x=\rho(x_\lambda)\cdot g$ for some $g\in\As(Q)$, and hence
$x\in O(x_\lambda)$ or $x\in O(\rho(x_\lambda))$, from which the proposition easily follows.
\end{proof}
\begin{thm}\label{thm:emeddable}
Let $(Q, \rho)$ be a symmetric quandle.
Then $(Q, \rho)$ is embeddable in $\As(Q,\rho)$ 
if and only if the underlying quandle $Q$ is embeddable in $\As(Q)$.
\end{thm}
\begin{proof}
Suppose that $(Q, \rho)$ is not embeddable in $\As(Q,\rho)$.
Then there exist distinct elements $x, y \in Q$ such that  
$s_x = s_y \in \As(Q, \rho)$.  
This implies that $e_x e_y^{-1} \in Z(Q, \rho)$,  
and hence $e_x e_y^{-1}$ can be uniquely expressed as  
\begin{equation}\label{eq:decomp}
e_x e_y^{-1} = 
\prod_{\lambda \in \Lambda} (e_{x_{\lambda}} e_{\rho(x_{\lambda})})^{c_{\lambda}}, 
\end{equation}
where $c_{\lambda} \in \mathbb{Z}$ and $c_{\lambda} = 0$ for all but finitely many $\lambda \in \Lambda$.  
In the abelianization $\As(Q)_{\mathrm{Ab}}$, we obtain  
\begin{align}\label{eq:embedd}
\begin{split}
[e_x][e_y]^{-1}& = 
\prod_{\lambda \in \Lambda} [e_{x_{\lambda}} e_{\rho(x_{\lambda})}]^{c_{\lambda}}  \\
&=
\prod_{\lambda \in \Lambda_1} [e_{x_{\lambda}}]^{2c_{\lambda}} 
\prod_{\mu \in \Lambda_2} [e_{x_{\mu}}]^{c_{\mu}} 
\prod_{\mu \in \Lambda_2} [e_{\rho(x_{\mu})}]^{c_{\mu}}.
\end{split}
\end{align}
Here we should recall that $\As(Q)_{\mathrm{Ab}}$ is a free abelian group and, by Lemma \ref{lem:orbits}, 
the following subset is linearly independent:
\[
\{[e_{x_{\lambda}}]\mid \lambda\in\Lambda=\Lambda_{1}\sqcup\Lambda_{2}\}\cup
 \{[e_{\rho(x_{\mu})}]\mid \mu\in\Lambda_2\}\subset\As(Q)_{\mathrm{Ab}}.
 \]
It follows from Proposition~\ref{pro:asq-abelianization} and Proposition~\ref{lem:embedd-orbit} 
that exactly one of the following three conditions holds in $\As(Q)_{\mathrm{Ab}}$:
\begin{enumerate}
\item $[e_x] = [e_{x_{\lambda}}]$ for some $\lambda \in \Lambda_1$,  
\item $[e_x] = [e_{x_{\mu}}]$ for some $\mu \in \Lambda_2$,  
\item $[e_x] = [e_{\rho(x_{\mu})}]$ for some $\mu \in \Lambda_2$.
\end{enumerate}
Suppose that $[e_x] \neq [e_y]$. 
We claim that none of the cases (1), (2), or (3) can occur.  
By comparing both sides of the equation~\eqref{eq:embedd}, 
we see that:
If (1) holds, then $2c_{\lambda} = 1$, which contradicts $c_{\lambda} \in \mathbb{Z}$.
If (2) holds, 
then we would simultaneously obtain both
 $c_{\mu} = 1$ and $c_{\mu} = 0$,
or both  $c_{\mu} = 1$ and $c_{\mu} = -1$,
depending on whether $[e_y]\not=[e_{\rho(x_{\mu})}]$
or $[e_y]=[e_{\rho(x_{\mu})}]$, 
respectively.  
Similarly, if (3) holds, 
then we would simultaneously obtain both
 $c_{\mu} = 1$ and $c_{\mu} = 0$,
or both  $c_{\mu} = 1$ and $c_{\mu} = -1$,
depending on whether $[e_y]\not=[e_{x_{\mu}}]$
or $[e_y]=[e_{x_{\mu}}]$, 
respectively.  
In either case, we again reach a contradiction.

Therefore, we conclude that $[e_x] = [e_y]$ and hence $[e_x][e_y]^{-1}=1$, 
which implies that $c_{\lambda} = 0$ for all $\lambda \in \Lambda$ in \eqref{eq:embedd}.  
In view of the equation~\eqref{eq:decomp}, it follows that $e_x e_y^{-1} = 1$, and hence $e_x = e_y$.  
Thus we conclude that the underlying quandle $Q$ is not embeddable.
\end{proof}

\begin{ack}
The first author was partially supported by JSPS
KAKENHI Grant Number 20K03600 and 24K06727.
\end{ack}

\end{document}